\documentclass[a4paper]{amsart}
\usepackage[latin1]{inputenc}
\usepackage{amsmath}
\usepackage{amssymb}
\usepackage{amsthm}

\usepackage[colorlinks,urlcolor=blue]{hyperref}

\newcommand{\N}{\mathbb{N}}
\renewcommand{\Pr}{\mathbb{P}}
\newcommand{\R}{\mathbb{R}}

\newcommand{\Bij}{B_i^j}
\newcommand{\Cij}{C_i^j}
\newcommand{\Tij}{T_i^j}
\newcommand{\sij}{s_i^j}
\newcommand{\uij}{u_i^j}
\newcommand{\xij}{x_i^j}

\newcommand{\covering}{\mathcal{C}}
\newcommand{\CC}{\mathcal{C}}
\newcommand{\HH}{\mathcal{H}}

\newcommand{\PP}{\mathcal{P}}
\newcommand{\Q}{\mathcal{P}}
\newcommand{\QQ}{\mathcal{P}}
\newcommand{\graphs}{\mathcal{U}}

\newcommand{\comp}{\mathcal{CG}}

\newcommand{\CI}{\mathcal{CVX}}

\newcommand{\DC}{\mathcal{DC}}
\newcommand{\incomp}{\mathcal{ICG}}

\renewcommand{\NG}{\mathcal{NG}}
\newcommand{\OSG}{\mathcal{OSG}}
\newcommand{\SG}{\mathcal{SG}}

\newcommand{\abs}[1]{\lvert #1 \rvert}
\newcommand{\deltacut}{\delta_{\square}}

\newcommand{\limitsx}[1]{\widehat{#1}}
\newcommand{\maxlimits}[1]{\widehat{#1}^*}

\renewcommand\P{\mathcal{P}}

\DeclareMathOperator{\col}{col}
\DeclareMathOperator{\Ent}{Ent}
\DeclareMathOperator{\Forb}{Forb}

\numberwithin{equation}{section}

\newtheorem{theorem}{Theorem}[section]
\newtheorem{lemma}[theorem]{Lemma}

\newtheorem{conjecture}[theorem]{Conjecture}

\theoremstyle{definition}
\newtheorem{definition}[theorem]{Definition}

\theoremstyle{remark}
\newtheorem{remark}[theorem]{Remark}

\newenvironment{romenumerate}[1][0pt]{
\addtolength{\leftmargini}{#1}\begin{enumerate}
 }{\end{enumerate}}

\newcommand\bigpar[1]{\bigl(#1\bigr)}

\newcommand\set[1]{\ensuremath{\{#1\}}}
\newcommand\bigset[1]{\ensuremath{\bigl\{#1\bigr\}}}
\newcommand\Bigset[1]{\ensuremath{\Bigl\{#1\Bigr\}}}

\newcommand\oi{[0,1]}
\newcommand\ooi{[0,1)}
\newcommand\punkt{.\spacefactor=1000}    

\newcommand\eg{e.g\punkt}

\newcommand{\aex}{a.e\punkt}
\newcounter{jeppe}
\newcommand\Claim[1]{\refstepcounter{jeppe}%
  \smallskip\noindent\emph{Claim \textup{\arabic{jeppe}:} #1}}
\newcommand\ga{\alpha}
\newcommand\cA{\mathcal A}
\newcommand\cI{\mathcal I}
\newcommand\cP{\mathcal P}
\newcommand\tC{\tilde C}
\newcommand\gG{\Gamma}
\newcommand\bbN{\mathbb N}
\newcommand{\refT}[1]{Theorem~\ref{#1}}

\newcommand{\refL}[1]{Lemma~\ref{#1}}

\newcommand\floor[1]{\lfloor#1\rfloor}

\newcommand{\tind}{p}
\newcommand\Psiq{\Psi^*}

\newcommand\psiqgx[1]{\Psiq_{G,#1}}

\newcommand\psiqgw{\psiqgx W}
\newcommand\psiqfw{\psiqgw}

\newcommand\xxn{x_1,\dots,x_n}
\newcommand\xxm{\xxn}
\newcommand\xx{\mathbf{x}}
\newcommand\hcP{\widehat{\cP}}
\newcommand\ffree{$G$-free}
\newcommand\qij{_{ij}}
\newcommand\QE{E}
\newcommand\setoi{\set{0,1}}
\newcommand{\tend}{\longrightarrow}
\newcommand\pto{\overset{\mathrm{p}}{\tend}}

\newcommand\dD{\partial D}
\begingroup
  \divide\count255 by 60
  \multiply\count255 by -60
  \advance\count255 by \time
  \ifnum \count255 < 10 \xdef\klockan{\the\count1.0\the\count255}
  \else\xdef\klockan{\the\count1.\the\count255}\fi
\endgroup

\title{On String Graph Limits and the Structure of a Typical String Graph}
\author{Svante Janson \and Andrew J. Uzzell}
\address{Department of Mathematics, Uppsala University, P.O.~Box 480, SE-751 06 Uppsala, Sweden}
\email{\href{mailto:svante.janson@math.uu.se}{svante.janson@math.uu.se}}
\email{\href{mailto:andrew.uzzell@math.uu.se}{andrew.uzzell@math.uu.se}}
\date{12 March, 2014}

\thanks{Partly supported by the Knut and Alice Wallenberg Foundation.}

\begin{document}

\begin{abstract}
We study limits of convergent sequences of string graphs, that is, graphs with an intersection representation consisting of curves in the plane.  We use these results to study the limiting behavior of a sequence of random string graphs.  We also prove similar results for several related graph classes.
\end{abstract}

\maketitle

\section{Introduction}\label{se:intro}

Given a graph property~$\PP$, it is interesting to study the structure of a typical graph that satisfies $\PP$.  A natural definition of a ``typical'' graph is a graph chosen uniformly at random from all graphs of a given order that satisfy $\PP$.  One can choose a sequence of random graphs in this way and then study its limiting behavior.  The theory of graph limits concerns the asymptotic behavior of certain sequences of graphs, and therefore provides a natural framework for studying the structure of a typical graph with a given property.  In this paper, we will study the structure of string graphs.  We will study graph limits of string graphs and of related graph classes and draw some conclusions about random string graphs and random elements of these other classes.  

A \emph{planar curve} is the image of a continuous function~$f : [0,1] \to
\R^2$.  The points $f(0)$ and~$f(1)$ are called the \emph{endpoints} of the
curve.  A \emph{string representation} of a graph~$G$ is a collection of
planar curves~$\{A_v : v \in V(G)\}$ such that $A_u \cap A_v \neq \emptyset$
if and only if $uv \in E(G)$.   We say that a graph $G$ is a \emph{string graph} if it has a string
representation and we let $\SG$ denote the family of string graphs.  String graphs have been studied by many authors, see e.g.\ 
\cite{Sinden,EET,KGK,PT06}
and the further references given there.

It is intuitively clear that a graph has a string representation if and only
if it has an intersection representation consisting of arcwise-connected
sets. Alternatively, we may assume that the curves in the definition are
homeomorphic images of $\oi$, and there are several other variations of the
definition that give the same class of graphs.
Although such equivalences are well-known, 
we have not found a detailed proof of these equivalences in the
literature, so we give a 
proof in
Appendix~\ref{se:arcwiserep}.

As mentioned above, we will also study several special classes of string graphs.  First, an \emph{outer-string representation} of a
graph~$G$ is a string representation such that all of the curves~$A_v$ lie
in a disk and such that each $A_v$ has an endpoint on the boundary of the
disk.  We say that a graph $G$ is an \emph{outer-string graph} if it has an outer-string
representation and let $\OSG$ denote the family of
outer-string graphs.  (Outer-string graphs were first so called
in~\cite{Kra91}, but were studied in the monograph~\cite{KGK}, there denoted~$Ng^0$.
Sinden~\cite{Sinden} studied a special case in which the strings are
required to meet the boundary of the disk in a prescribed order.) 
It is clear from the definition that every outer-string
graph is a string graph.  
It was shown in \cite{KGK} that
the converse does not hold; 
one consequence of our results is that only a very small fraction
of all string graphs are outer-string graphs
(see Remark~\ref{re:proper}). 

Next, we consider the class of graphs with an intersection representation consisting of strings lying between two parallel segments, with one endpoint on each segment.  It has been discovered several times~\cite{GRU83,KGK,Lov83} that this class is equal to the class of incomparability graphs.  (Recall that if $<$ is a partial order on $[n]$, then the \emph{incomparability graph} of~$<$ is the graph with vertex set~$[n]$ in which $x \sim y$ if and only if neither $x < y$ nor $y < x$.)  We let $\incomp$ denote the class of incomparability graphs.

Finally, we say that a graph is a \emph{two-clique graph} if it is the disjoint union of at most two cliques.
We let $\mathcal{TCG}$ denote the class of two-clique graphs.

It is easy to see that all of these classes are hereditary.  They are
related by the following theorem~\cite[Theorem 5.8]{KGK}.  For $k \geq 1$, the authors of~\cite{KGK}
defined $k\mathcal{NG}$ to be the set of graphs $G$ with the property that if
$H_1$, \dots, $H_k$ are cliques such that the vertex sets of~$G$ and of the $H_i$ are all pairwise disjoint,
then any graph~$F$ with $V(F) = V(G \cup
H_1 \cup \dots \cup H_k)$ and $E(F) \supseteq E(G \cup H_1 \cup \dots \cup
H_k)$ is a string graph.

\begin{theorem}[\cite{KGK}]\label{th:classes}
We have the following characterizations of the classes~$k\NG$.
\begin{romenumerate}
\item $1\NG$ equals the class~$\OSG$ of outer-string graphs;
\item $2\NG$ equals the class~$\incomp$ of incomparability graphs;
\item $3\NG$ equals the class~$\mathcal{TCG}$ of two-clique graphs.
\end{romenumerate}
Moreover, $4\NG = \set{K_0}$, where $K_0$ denotes the null graph, and $k\NG = \emptyset$ for $k \geq 5$.
\qed
\end{theorem}

It is easy to see that $k\NG \subseteq (k - 1)\NG$ for all $k \geq 2$, and it is shown in~\cite{KGK} that (for $k \leq 5$) all of these inclusions are strict.

\subsection{Preliminaries}\label{se:prelims}

As noted above, we will use the theory of graph limits to study these graph
classes.  For the basic notions of graph limits, see the recent book 
by Lov\'asz \cite{LovaszBook}
or, e.g., \cite{BCLSV1,DJ08,LS06}.
We recall that certain sequences of graphs~$(G_n)$, with
$|V(G_n)|\to\infty$, are defined to be \emph{convergent}. A convergent
sequence of graphs has a limit, a \emph{graph limit}; these objects can be
defined in several different ways.
A \emph{graphon} is a (measurable) symmetric function $\oi^2\to\oi$.
Each graphon~$W$ defines a unique graph limit~$\Gamma$
(we say that $W$ \emph{represents} $\Gamma$), and every graph limit is
represented by some graphon; however, the representing graphon is not
unique. We say that two graphons $W$ and~$W'$ are \emph{equivalent},
and write $W\cong W'$, 
if they represent the same graph limit. (Hence, the graph limits correspond to
equivalence classes of graphons, and may be defined in this way.)

The \emph{entropy} of a graphon~$W$ is defined as
\begin{equation*}
    \label{entw}
\Ent(W) = \int_0^1 \int_0^1  h(W(x,y)) \,d x \,d y,
\end{equation*}
where $h(x) = -x \log_2(x)-(1-x)\log_2(1-x)$ is the usual binary entropy
function, see \cite{HJS} and \cite{ChatterjeeVaradhan11}. 
The entropy $\Ent(\Gamma)$
of a graph limit $\Gamma$ is the entropy of any representing graphon;
equivalent graphons have the same entropy, so this defines $\Ent(\Gamma)$
uniquely.

If $\P$ is a graph class (or graph property, we do not distinguish between
these), then $\limitsx \P$ denotes the set of all graph limits $\Gamma$ such that
there exists a sequence~$G_n\in\P$ with $G_n\to \Gamma$. 
We will also, slightly abusing the notation, let $\limitsx \P$ denote the set of
all graphons that represent such a graph limit.
Furthermore, let $\maxlimits{\P}$ denote the set of graph limits (or graphons)
in $\limitsx\P$ with
maximal entropy. (This set is nonempty, except in the trivial case when 
the graph class $\P$ is finite, see \cite{HJS}.)
For the importance of the set $\maxlimits{\P}$ of maximum-entropy graph limits,
see \cite{HJS}.

We define some special graphons and sets of graphons, see further
\cite{HJS}. 
For convenience, we
define them on $\ooi^2$ instead of $\oi^2$; this is clearly immaterial.

\begin{definition}
Fix $k \geq 1$.  For each $i \in [k]$, let $I_i = [(i - 1)/k, i/k)$.  
  \begin{romenumerate}[-6pt]
  \item 
Let $R_k$ be the set of all graphons $W$ such that $W(x,y)=1/2$ 
on $\bigcup_{i\neq j} I_i\times I_j$ and $W(x,y)\in\set{0,1}$ on each
$I_i\times I_i$.  We also let $R_{\infty}$ consist of the constant graphon~$W = 1/2$.
\item 
For $s \in \{0, \dots, k\}$, let $W^*_{k,s}$ be the graphon that is $1/2$
on $\bigcup_{i\neq j} I_i\times I_j$,
$1$ on each $I_i\times I_i$, $i\le s$, and $0$ on each $I_i\times I_i$, $i>s$.
\item 
For $a\in\oi$, let $W^k_a$ denote the graphon obtained from $W^*_{k,k}$ by 
changing it on $I_1 \times I_1$ 
such that
$W^k_a = 1$ on $[0, a/k)^2 \cup [a/k, 1)^2$ and $W^k_a = 0$ on $[0, a/k)
	  \times [a/k, 1) \cup [a/k, 1) \times [0, a/k)$. 
  \end{romenumerate}
\end{definition}

Note that $W^*_{k,s}\in R_k$ for $s=0,\dots,k$, and $W^k_a\in R_k$.
Furthermore, 
$\Ent(W)=1-1/k$ for every $W\in R_k$;
in particular, $\Ent(W^*_{k,s})=\Ent(W^k_a)=1-1/k$.

We note that $W^k_0=W^k_1=W^*_{k,k}$ and that
$W^k_a \cong W^k_{1-a}$.  However, if $b \notin \{a, 1- a\}$,
then $W^k_a \ncong W^k_{b}$, because the two graphons have different edge
densities.  Indeed, for any~$a \in [0, 1]$, we have 
\begin{equation}\label{eq:Wdensity}
\int_0^1\int_0^1 W^k_a 
= \dfrac{a^2 + (1-a)^2}{k^2} + \dfrac{k-1}{k^2} + 
\frac{k(k-1)}{2k^2}.
\end{equation}

Given $t \geq 1$ and $0 \leq s \leq t$, we define $\CC(t, s)$ to be the set
of graphs whose vertex sets can be partitioned into $s$ (possibly empty)
cliques and $t-s$ (possibly empty) independent sets.  In particular, $\CC(t, 0)$ is the class of $t$-colorable graphs.  If $\Q$ is a hereditary
property, the \emph{coloring number} of~$\Q$, denoted $\col(\Q)$,
is the largest $t$ for which $\mathcal{C}(t,s) \subseteq \Q$ for some $s
\leq t$, see e.g.\ the survey \cite{Bol07}.  
(We define $\col(\Q) = \infty$ if $\Q$ is the class of
all unlabeled finite graphs; otherwise, $\col(\Q)$ is finite.)

Let $\graphs$ denote the set of unlabeled finite graphs and let $\graphs_n$
denote the set of unlabeled graphs on $n$ vertices.  Given a graph
property~$\PP$, let $\PP_n = \PP \cap \graphs_n$ and let $\PP^L_n$ denote the set of labeled graphs in $\PP$ with vertex set~$[n]$. The function $n \mapsto \abs{\PP_n}$ is called the (unlabeled) \emph{speed} of~$\PP$. 
(The labeled speed is defined similarly.
For our purposes, it does not matter whether we consider labeled or
unlabeled graphs, since the difference is at most a factor $n!=2^{o(n^2)}$,
  which will be negligible.)
The speed of graph properties has been studied in many papers,
see e.g.\ \cite{Ale92,BT,Bol07,PT06}.

Hatami, Janson, and Szegedy~\cite[Theorem 1.9]{HJS} proved the following
result that relates the maximum-entropy graph limits of a hereditary
property~$\Q$, the speed of~$\Q$, and the coloring number of~$\Q$. 

\begin{theorem}[\cite{HJS}]\label{th:maxlimits}
If $\Q$ is a hereditary class of graphs, there exists $r \in \N \cup \set{\infty}$ such that $\max_{\Gamma \in \limitsx{\Q}} \Ent(\Gamma) = 1 - 1/r$ and every
graph limit $\Gamma \in \maxlimits{\Q}$ can be represented by a graphon~$W
\in R_r$; hence $\maxlimits{\Q} = \limitsx{\Q}\cap R_r$.  
 Moreover,
\begin{align*}
	r 	  &= \sup \{t \, : \, W^*_{t,s} \in \limitsx{\Q} \text{ for some $s \leq t$}\}\\
&= \sup \{t \, : \, \mathcal{C}(t,s) \subseteq \Q \text{ for some $s \leq t$}\}\\
&= \col(\PP).
\end{align*}
Moreover,
\begin{equation}\label{eq:speed}
  \abs{\PP_n} = 2^{\bigl(1 - \frac{1}{r} + o(1)\bigr)\binom{n}{2}}.
\end{equation}
\end{theorem}

We note that~\eqref{eq:speed} was originally proved independently by Alekseev~\cite{Ale92} and by Bollob\'as and Thomason~\cite{BT}.

\subsection{Main Results}\label{se:results}

One of our aims is to classify maximum-entropy graph limits of string graphs and of related families of graphs.  In order to do so, we prove a somewhat more general result about the maximum-entropy graph limits of certain hereditary properties.  Before we can state this result, we need to define a few special graphs.

Let $k \geq 3$.  We define three graphs, each with vertex set consisting of
the $k+\binom k2$ subsets of~$[k]$ of size either one or~two.  We denote the
vertex 
corresponding to a subset~$S$ by $v_S$.
First, let $G_k$ be
the intersection graph of this family of subsets, that is, the graph in
which two vertices are adjacent if and only if their corresponding subsets
have non-empty intersection.  Second, we define $B_k$ in the same way,
except that we do not allow edges between two vertices that correspond to
subsets of size~two.  Hence, $B_k$ is bipartite.  Finally, for $k \geq 4$, we define $H_k$
be the subgraph of $G_k$ with
$E(H_k) = E(G_k) \setminus \{v_{\set{1,2}}v_{\set{1,3}}, 
v_{\set{1,2}}v_{\set{2,3}}, v_{\set{1,3}}v_{\set{2,3} }\}$.
(Equivalently, $H_k$ is such that $v_{\{s\}} \sim v_{\{t, u\}}$ if and only if $s = t$
or $s = u$, while $v_{\{s_1, s_2\}} \sim v_{\{t_1, t_2\}}$ if and only if
$\{s_1, s_2\} \cap \{t_1, t_2\} \neq \emptyset$ and $\max\{s_1, s_2, t_1,
t_2\} \geq 4$.)

Given a family of graphs~$\mathcal{H}$, we define $\Forb^*(\HH)$ to be the class of graphs that do not contain a copy of any $H \in \HH$ as an induced subgraph.  It is easy to see that $\Forb^*(\HH)$ is a hereditary class.

\begin{theorem}\label{th:general}
Let $\PP$ be a hereditary class of graphs and let $r = \col(\PP)$.  
If\/ $3 \leq r < \infty$ and $\PP \subseteq \Forb^*(\{G_{r + 1}, B_{r+1}, H_{r+1}\})$, then
\[
  \maxlimits{\PP} \subseteq \bigl\{ W^r_a \, : \, a \in [0, 1] \bigr\}
\]
up to equivalence of graphons.

If $r = 2$ and $\PP \subseteq \Forb^*(\{G_{3}, B_{3}\})$, then
\[
  \maxlimits{\PP} \subseteq \bigl\{ W^2_a \, : \, a \in [0, 1] \bigr\}
\]
up to equivalence of graphons.
\end{theorem}

Note that the assumption $\PP \subseteq \Forb^*(\{G_{r + 1}, B_{r+1},
H_{r+1}\})$ is equivalent to 
$G_{r + 1},\allowbreak B_{r+1}, H_{r+1}\notin \PP$.

Using Theorem~\ref{th:general}, we can characterize maximum-entropy limits of string graphs, outer-string graphs, and incomparability graphs as follows.  See also Theorem \ref{th:convexlimits} for another related result.

\begin{theorem}\label{th:stringlimits}
Up to equivalence of graphons,
\[
	\maxlimits{\SG} = \bigl\{ W^4_a \, : \, a \in [0, 1] \bigr\}.
\]
\end{theorem}

We note that the inclusion $\supseteq$ in \refT{th:stringlimits} was shown in
\cite[Example 2.4]{HJS}, but the question of equality was left open there.

\begin{theorem}\label{th:outerstringlimits}
Up to equivalence of graphons,
\[
	\maxlimits{\OSG} = \bigl\{ W^3_a \, : \, a \in [0, 1] \bigr\}.
\]
\end{theorem}

\begin{theorem}\label{th:incomplimits}
Up to equivalence of graphons,
\[
	\maxlimits{\incomp} = \bigl\{ W^2_a \, : \, a \in [0, 1] \bigr\}.
\]
\end{theorem}

\begin{theorem}\label{th:2cliquelimits}
Up to equivalence of graphons,
\[
	\maxlimits{\mathcal{TCG}} = \bigl\{ W^1_a \, : \, a \in [0, 1] \bigr\}.
\]
\end{theorem}

The proof of Theorem~\ref{th:2cliquelimits} is trivial, so we omit it.
(The other theorems are proved later.) 
In fact, it is  easy to see that 
$\limitsx{\mathcal{TCG}} = \maxlimits{\mathcal{TCG}} = 
\bigl\{ W^1_a \, : \, a \in [0, 1] \bigr\}$.
For the graph classes $\P$
in Theorems \ref{th:stringlimits}--\ref{th:incomplimits},
$\limitsx\P\supsetneq\maxlimits\P$, i.e., these classes have graph limits
that do not have maximum entropy. We leave it as an open problem to classify all
graph limits for these classes.

\subsection{Results from Graph Limit Theory}\label{se:tools}

Now we assemble the tools from graph limit theory that we will need in order to prove the results in Section~\ref{se:results}.


Let $W : [0, 1]^2 \to [0, 1]$ be a graphon and let $G$ 
be  a graph with vertex set~$[n]$.
We define 
\begin{equation}\label{psiqfw}
  \psiqgw(\xxn)=\prod_{ij\in E(G)} W(x_i,x_j) 
\prod_{ij\not\in E(G)}\bigpar{1- W(x_i,x_j)}
\end{equation}
and recall that the induced subgraph density of $G$ in $W$ is defined as
\begin{equation}\label{tindfw}
  \tind(G,W)=\int_{\oi^n}\psiqgw.
\end{equation}

We further say that the graph $G$
 is \emph{$W$-constructible} if there exist (not
necessarily distinct) points $x_1$, \dots, $x_n \in [0, 1]$ such that 
$\psiqgw(x_1,\dots,x_n)>0$, i.e., more explicitly,
\begin{align}
W(x_i, x_j) = 0 \quad &\Longrightarrow \quad ij \notin E(G) \label{eq:Wconst0}\\
\intertext{and}
W(x_i, x_j) = 1 \quad &\Longrightarrow \quad ij \in E(G). \label{eq:Wconst1}
\end{align}
(If $0 < W(x_i, x_j) < 1$, then there is no restriction on~$ij$.)  If
$\mathbf{x} = (x_1, \dots, x_n)$ is such that \eqref{eq:Wconst0}
and~\eqref{eq:Wconst1} hold, 
i.e.\ $\psiqgw(\xx)>0$,
then we say that $\mathbf{x}$ is a
\emph{witnessing vector for $G$}.

Given a graphon~$W$, there is a standard definition
of a $W$-random graph $G(n, W)$: let $X_1, X_2, \dots$ be an i.i.d.~sequence of uniform random variables in $\oi$.  For each~$n$, $G(n, W)$ is a graph with vertex set~$[n]$ in which the edge~$ij$ is present with probability~$W(X_i, X_j)$, independently of all other edges.  By~\cite[Proposition 11.32]{LovaszBook}, $G(n, W) \to W$ almost surely.

\begin{remark}
It follows from~\eqref{tindfw} that
for every graph~$G$ with vertex set $[n]$,
\begin{equation*}\label{eq:Wdistrib}
	\Pr(G(n, W) = G) = p(G; W).
\end{equation*}
Hence,
\begin{equation}\label{eq:0as}
	\Pr(G(n, W) = G) = 0
\iff \tind(G,W)=0
\iff
  \psiqgw(\xx)=0
\quad\text{for a.e.\ } \xx.
\end{equation}
This is thus a condition for a.e.\ $\xx$, while we have defined $G$ to be
not $W$-constructible if $\psiqgw(\xx)=0$ for \emph{every} $\xx$. This looks
like a minor technical difference, and indeed it is 
(although this turns out to be non-trivial to prove).
Petrov's general removal lemma \cite[Theorem 1]{petrov}
shows that if the properties in
\eqref{eq:0as} hold, then $W$ can be modified on a null set such that then
$\psiqgw=0$ everywhere, i.e., $G$ is not $W$-constructible.
(Note that the properties in \eqref{eq:0as} are preserved if $W$ is modified
on a null set, and more generally if $W$ is replaced by an equivalent
graphon, but the property 
that $G$ is (not) $W$-constructible is not.)  
\end{remark}

Let $\cP$ be a hereditary class.
It is easily seen, see \cite[Theorem 3.3]{Jan13b}, that if $W$ is a graphon,
then $W\in\hcP$ if and only if $\tind(G;W)=0$ for every graph $G\notin\cP$.
In particular, see
Diaconis, Holmes, and Janson~\cite[Theorem 3.2]{DHJ},
if $\mathcal{H}$ is a set of graphs and $\Q = \Forb^*(\mathcal{H})$,
then $\Gamma \in \limitsx{\Q}$ if and only if $p(H; \Gamma) = 0$ for all $H
\in \mathcal{H}$. 
We will need the following extension of this,
which is a consequence of 
Petrov's removal lemma~\cite[Theorem 1]{petrov},
see Section \ref{se:removal}.

\begin{lemma}\label{le:magic}
Let $\PP$ be a hereditary class of graphs.
If\/
$\Gamma \in \widehat{\PP}$, then there exists a graphon~$W$ representing
$\Gamma$ such that if a graph~$G$ 
is
$W$-constructible, then $G \in \PP$.  
Moreover, if\/ $\Gamma \in
\maxlimits{\PP}$  and  $r = \col(\PP)$, 
then there exists~$W \in R_r$ such that if $G$ is $W$-constructible, then $G \in \PP$.  
\end{lemma}

We conclude this section with a special class of graphons
that will be needed in the proof of Theorem~\ref{th:general}. 
A \emph{disjoint clique graphon} is a graphon of the form~$W=\sum_{\ga\in \cA} \mathbf{1}_{A_\ga \times A_\ga}$
for a family~$(A_\ga)_{\ga\in\cA}$  of non-empty
pairwise disjoint  subsets of~$\oi$.
Here the index set $\cA$ may be finite, countably infinite or uncountable.
We say that $W$ has \emph{parts} $A_\ga$, $\ga\in\cA$.

Since a graphon is assumed to be (Lebesgue) measurable,
the section $\set{y:W(x,y)=1}$ is measurable for a.e.\ $x\in\oi$.
In particular, if $W$ is a disjoint clique graphon and $A_\ga$ one of its
parts, then either $A_\ga$ is a null set, and thus measurable, or there
exists $x\in A_\ga$ such that
$\set{y:W(x,y)=1}$ is measurable. Since the latter set equals $A_\ga$, we
see that in both cases $A_\ga$ is measurable. In other words, the parts of a
disjoint clique graphon are all measurable.

\begin{remark}
  It is easy to see that up to equivalence, we may eliminate all parts with
  measure 0, leaving only a countable number of parts $A_n$. 
Moreover, up to
  equivalence, a disjoint clique graphon is uniquely determined by the 
sequence $|A_n|$ of measures of these parts, arranged in (weakly) decreasing
order. It is also easy to see that if
$\DC$ is the family of 
\emph{disjoint clique graphs}, i.e.~graphs that are disjoint unions of cliques,
then the set of all disjoint clique graphons 
equals
$\limitsx{\DC}$.
See further 
Janson~\cite[Section 7]{Jan13b}.
\end{remark}

The rest of this paper is organized as follows.
In Section~\ref{se:proofgeneral}, we prove
Theorem~\ref{th:general}.  In Sections \ref{se:stringproof},~\ref{se:outerstringproof}, and~\ref{se:incompproof}, we prove Theorems \ref{th:stringlimits},~\ref{th:outerstringlimits}, and~\ref{th:incomplimits}, respectively.  In Section~\ref{se:random},
we derive one more result each about the structure of typical string graphs 
and the structure of typical outer-string graphs, and make several
conjectures.  Finally, in Section~\ref{se:removal}, we prove
Lemma~\ref{le:magic}.

\section{Proof of Theorem \ref{th:general}}\label{se:proofgeneral}

The plan of the proof of Theorem~\ref{th:general} is straightforward.  We
want to show that if $\PP$ is as in the statement of the theorem, then any
maximum-entropy element~$\Gamma$ of~$\widehat{\PP}$ can be represented by a
graphon~$W$ that has a number of desirable properties, eventually leading to
the conclusion that $W\cong W_a^r$ for some $a$.
We use Lemma~\ref{le:magic} to find a ``nice'' version of $W$; then, 
in each case, we
will show that if $W$ does not have the desired property, then
this implies that $\PP$ contains at least~one of the graphs
$G_{r+1}$,~$B_{r+1}$, and~$H_{r+1}$, which is a contradiction. 

\begin{proof}[Proof of Theorem~\ref{th:general}.]
Suppose that $\Gamma \in \maxlimits{\PP}$ has entropy~$1 - 1/r$
and let $W \in R_r$ be the graphon 
representing $\Gamma$ whose existence is guaranteed by Lemma~\ref{le:magic}.
Thus, every graph that is
$W$-constructible belongs to $\PP$. 
In particular, our assumption on $\PP$ implies that $G_{r+1}$, $B_{r+1}$, and
$H_{r+1}$ are \emph{not} $W$-constructible. 
For (notational) convenience, we let all graphons be defined on $\ooi$ in
this section.

Let $W_i$ denote the restriction of~$W$ to~$I_i \times I_i$.  
By rescaling the interval~$I_i$ to~$\ooi$, we may regard $W_i$ as a graphon.
Recall that by the definition of $R_r$, $W_i$ takes only the values 0 and 1.

\Claim{At most one of the sets~$I_i$ can contain some point~$x$ such that
  $W(x,x) = 0$.}\label{cl00}  
If not, then without loss of generality there exist $x \in I_1$ and $y \in
I_2$ such that $W(x,x) = W(y,y) = 0$.  Let $F$ be a bipartite graph.  If
$(A, B)$ is a bipartition of~$V(F)$ with $\abs{A} = n_1$ and $\abs{B} =
n_2$, then we can construct a witnessing vector for $F$ by choosing $n_1$
copies of~$x$ and $n_2$ copies of~$y$.  Thus, we may conclude that $F$ is
$W$-constructible.   In particular, $B_{r+1}$ is $W$-constructible, 
which is a contradiction. 

\Claim{$W = 1$ along the diagonal.}\label{cl1}
If not, then without loss of generality there exists~$z \in I_1$ such
that $W(z,z) = 0$. We then partition the vertex set of~$G_{r+1}$ as follows:
let $V_1 = \{1, 2, 3\}$, let $V_2 = \{\set{1,2}, \set{1,3}, \set{2,3}\}$,
and for 
$\ell = 3$, \dots, $r$, let 
$V_{\ell} = \{\set{\ell + 1, j} : 1 \leq j \leq \ell
+ 1\}$.  (Here, we write $\set{i, i}$ for $\{i\}$.)  Observe that each
of $V_2$, \dots, $V_r$ induces a complete graph, while $V_1$ induces an
empty graph.  For $j = 2$, \dots, $r$, let $v_j \in I_j$ be arbitrary.  
By Claim \ref{cl00}, for each~$j$, we have $W(v_j, v_j) = 1$.  Then we
can construct a witnessing vector for $G_{r+1}$ by choosing three copies
of~$z$ and $j + 1$ copies of the point $v_j$ for $j = 2$, \dots, $r$. 
Thus, $G_{r+1}$ is $W$-constructible, 
which is a contradiction.

\Claim{If\/ $x,y,z\in I_1$ and $W(x,y)=W(y,z)=1$, then $W(x,z)=1$.}
\label{cl111}
We partition the vertex set of~$G_{r+1}$ in a different way.
Let $U_1 = \{1, 2, \set{1,2}\}$ and for $\ell = 2$, \dots, $r$, let
$U_{\ell} = \{\set{\ell + 1, j} : 1 \leq j \leq \ell + 1\}$.  Observe that
$U_1$ induces a copy of~$P_3$ and that each of $U_2$, \dots, $U_r$ induces a
complete graph.  
Suppose that there exist $x,y,z \in I_1$ such
that $W(x, y) = W(y, z) = 1$ but $W(x, z) = 0$.  For $\ell = 2$, \dots, $r$,
let $u_{\ell} \in I_{\ell}$ be arbitrary and note that, by Claim \ref{cl1},
 $W(u_{\ell}, u_{\ell}) = 1$.  Then the vector
consisting of $x$,~$y$,~$z$, and $\ell + 1$ copies of each point~$u_{\ell}$
is a witnessing vector for $G_{r+1}$, 
which implies that $G_{r+1}$ is $W$-constructible.
This contradiction proves the claim.

\Claim{Each $W_i$ is a disjoint clique graphon 
$\sum_{\ga\in\cA_i} \mathbf{1}_{A_{i,\ga} \times A_{i,\ga}}$, 
rescaled to $I_i$.}
\label{clpartition}
It suffices to consider $i=1$.
The relation on $I_1$ defined by $x\equiv y$ if $W(x,y)=1$ is 
symmetric since $W$ is, 
reflexive by Claim \ref{cl1}, and transitive by Claim \ref{cl111}; hence it
is an equivalence relation.
Denote the equivalence classes by $(A_\ga)_{\ga\in\cA}$; then 
$W_1=\sum_{\ga\in\cA}  \mathbf{1}_{A_\ga \times A_\ga}$.

\Claim{Each $W_i$ has at most two parts $A_{i,\ga}$.}\label{cl2}
If not, then without loss of generality $W_1$ has at least three parts
$A_1,A_2,A_3$.
We partition
the vertex set of~$G_{r+1}$ using the partition~$(V_1, \dots, V_r)$ defined
in the proof of Claim \ref{cl1} above.  
Let $a_i \in A_i$ for $i = 1$,~$2$,~$3$ and observe that 
$W(a_i,a_j) = \delta_{ij}$.  For $j = 2$, \dots, $r$, let $v_j \in I_j$ be
arbitrary.  Observe that the vector consisting of $a_1$, $a_2$, $a_3$, and
$j + 1$ copies of each point~$v_{j}$ is a witnessing vector for $G_{r+1}$.
Thus $G_{r+1}$ is $W$-constructible, which is a contradiction.

\Claim{$W = 1$ on at least $r - 1$ of the squares $I_i \times I_i$.}
\label{clr-1}
If not, there are at least two
restrictions~$W_i$ 
such that the partitions $(A_{i,\ga})$ in Claim \ref{clpartition}
have at least two parts, i.e., $|\cA_i|\ge2$.
Suppose without loss of generality that 
$|\cA_1|,|\cA_2|\ge2$, and let $A_1,A_2$ be parts of $W_1$ and $B_1,B_2$ be
parts of $W_2$.
First, suppose that $r \geq 3$.  We claim that
$H_{r+1}$ is $W$-constructible.  To show this, we partition $V(H_{r+1})$ as
follows.  Let $X'_1 = \{1, \set{1,3}\}$, let $X''_1 = \{\set{2,3}\}$, let $X'_2
= \{2, \set{1,2}\}$, and let $X''_2 = \{3\}$.  For $\ell = 3$, \dots, $r$,
let $X_{\ell} = \{\set{\ell + 1, j} : 1 \leq j \leq \ell + 1\}$. Observe that
the sets $X'_1$, $X''_1$, $X'_2$, $X''_2$, $X_3$, \dots, $X_r$ all induce
complete subgraphs, that there are no edges between $X'_1$ and $X''_1$, and
that there are no edges between $X'_2$ and $X''_2$.  Let $x'_1 \in A_1$, let
$x''_1 \in A_2$, let $x'_2 \in B_1$, let $x''_2 \in B_2$, and, for $\ell =
3$, \dots, $r$, let $x_{\ell} \in I_{\ell}$.  Let $\mathbf{x}$ denote the
vector consisting of two copies of~$x'_1$, one copy of~$x''_1$, two copies
of~$x'_2$, one copy of~$x''_2$, and  $\ell + 1$ copies of~$x_{\ell}$ for
$\ell = 3$, \dots, $r$.  Observe that $\mathbf{x}$ is a witnessing vector
for $H_{r+1}$.
Thus $H_{r+1}$ is $W$-constructible,
which is a contradiction.  If $r = 2$, we may repeat the argument above to show that $B_3$ is $W$-constructible, which is again a contradiction. This proves the claim.

To summarize, we have shown that if $\Gamma \in \maxlimits{\PP}$, then
$\Gamma$ may be represented by a graphon~$W \in R_r$ such that there exists
a measurable set~$A_1 \subseteq I_1$  such that $W = 1$ on $(A_1 \times A_1)
\cup ((I_1 \setminus A_1) \times (I_1 \setminus A_1))$, $W = 0$ on the rest
of $I_1 \times I_1$, and $W = 1$ on $\cup_{j=2}^r I_j \times I_j$.  That is,
$W \cong W^r_a$ for some~$a \in [0, 1]$, as claimed. 
\end{proof}

\section{Proof of Theorem~\ref{th:stringlimits}}\label{se:stringproof}

We begin by showing that the class of string graphs satisfies the hypotheses of Theorem~\ref{th:general}.

\begin{lemma}\label{le:stringforb}
	The graphs $G_5$,~$B_5$, and~$H_5$ are not string graphs.
\end{lemma}

\begin{proof}
The result for~$B_5$ has been discovered several
times~\cite{Sinden,EET,KGK}, while the result for $G_5$ was shown
in~\cite[Lemma 3.2]{PT06}.  The argument that we give is a slight
modification of the proof used in~\cite{PT06}.
(It actually works for any graph~$F$ with $B_5\subseteq F\subseteq G_{5}$.)

Suppose that any of $G_5$,~$B_5$, and~$H_5$ had a string representation
\set{A_i,A_{ij}}.
Then, 
we could select points $v_i \in A_i$ for each $i$ and~
points  $v_{ij} \in A_i \cap A_{ij}$, 
for each $i$ and~$j$ with $i\neq j$, 
Then,
for each $i$ and~$j$, we could let $e_{ij}$ be a curve consisting of the
portion of~$A_i$ between $v_i$ and $v_{ij}$, the portion of~$A_{ij}$ between
$v_{ij}$ and $v_{ji}$, and the portion of~$A_j$ between $v_{ji}$ and
$v_{j}$.  Because $e_{ij} \subseteq A_i \cup A_j \cup A_{ij}$, we have
$e_{ij} \cap e_{k\ell} = \emptyset$ if $i$, $j$, $k$, and $\ell$ are all
distinct.  Hence, the points~$v_i$ and the curves~$e_{ij}$ define a drawing
of~$K_5$ in which no two independent edges cross, which contradicts the
Hanani--Tutte theorem~\cite{Hanani,Tutte}.
\end{proof}

Let $G$ be a graph and let $\covering = \{V_1, \dots, V_k\}$ be a clique covering of~$G$.  We can define a graph~$G_{\covering}$ with $V(G_{\covering}) = [k]$ and $E(G_{\covering}) = \{ij : E(V_i, V_j) \neq \emptyset\}$.  We say that a graph~$G$ admits a \emph{planar clique covering} if there is a clique covering~$\covering$ of~$G$ such that $G_{\covering}$ is planar.  Similarly, we say that $G$ admits an \emph{outerplanar clique covering} if there is a clique covering~$\covering$ of~$G$ such that $G_{\covering}$ is outerplanar.

\begin{proof}[Proof of Theorem~\ref{th:stringlimits}.]
It is shown in~\cite{PT06} that $\col(\SG) = 4$ (see also Lemma~\ref{le:Ginclusion}).  This,
Lemma~\ref{le:stringforb}, and Theorem~\ref{th:general} imply that
$\maxlimits{\SG} \subseteq \{ W^4_a \, : \, a \in [0, 1] \}$ up to
equivalence of graphons. 

In the opposite direction, it is shown in~\cite[Example 2.4]{HJS} that
$W^4_{a} \in \maxlimits{\SG}$ for every $a\in\oi$.
For later use, we repeat the argument.

It is shown in~\cite[Corollary 2.7]{KGK} that $\CC(4,4) \subseteq \SG$ and
more generally in~\cite[Theorem 2.3]{KGK} that every graph that admits a
planar clique covering is a string graph.  (See also \cite[Example 2.4]{HJS}.)
Let $a \in [0,1]$ and, for each~$n$, let $G_n = G(n, W^4_a)$, where $G(n, W^4_a)$ is the $W^4_a$-random graph defined in Section~\ref{se:tools}.  Let $V'_1 = \set{i : X_i \in [0, a/4)}$, let $V''_1 = \set{i : X_i \in [a/4, 1/4)}$, and, for $j = 2, 3, 4$, let $V_j = \set{i : X_i \in I_j}$.  Observe that with probability~one, $V'_1$, $V''_1$, $V_2$, $V_3$, and $V_4$ are all cliques and there are no edges between $V'_1$ and $V''_1$.
This means that each $G_n$ almost surely admits a planar clique covering~$\covering$ with 
$G_{\covering} =  K_5 - e$, which means that each $G_n$ is a string graph.  Because $G_n \to W^4_a$ almost surely, it follows
that $W^4_a \in \limitsx{\SG} \cap R_4$, which, by
Theorem~\ref{th:maxlimits}, implies that $W^4_a \in
\maxlimits{SG}$, as claimed.  
This completes the proof.
\end{proof}

\begin{remark}\label{re:stringspeed}
Because $\col(\SG) = 4$, it follows from Theorem~\ref{th:maxlimits} that
\begin{equation}\label{eq:stringspeed}
	\abs{\SG_n} = 2^{\left(\frac{3}{4} + o(1)\right)\binom{n}{2}},  
\end{equation}
as proved by Pach and T\'oth \cite{PT06}.
\end{remark}

A graph~$G$ is a \emph{convex intersection graph} if it has an intersection
representation consisting of convex sets in the plane.  Let $\CI$ denote the
class of convex intersection graphs.
It is easy to see that $\CI$ is
hereditary and that $\CI \subseteq \SG$.  It is shown in~\cite[Proposition
  8.3.1]{KGK} that $\CI$ is in fact a proper subclass of~$\SG$.  
However, the next result shows that $\CI$ has the same asymptotic speed as the larger class~$\SG$ and furthermore shows that $\CI$ has the same maximum-entropy graph limits as $\SG$, i.e. $\maxlimits{\CI}=\maxlimits{\SG}$.
(The sets of all graph limits for these classes differ,
i.e., $\limitsx{\CI}\subsetneq\limitsx{\SG}$. For example, if $G\in\SG\setminus\CI$,
then the adjacency matrix of $G$ defines a graphon, and thus a graph limit,
that easily is seen to belong to
$\limitsx{\SG}\setminus\limitsx{\CI}$, cf.\ \cite[Section 3]{Jan13b} and
\cite[Proposition 4.10]{LovaszSzegedy:regularity}.)

\begin{theorem}\label{th:convexlimits}
We have
\begin{equation}\label{eq:convexcoloringnumber}
	\col(\CI) = 4
\end{equation}
and
\begin{equation}\label{eq:convexspeed}
	\abs{\CI_n} = 2^{\left(\frac{3}{4} + o(1)\right)\binom{n}{2}}.
\end{equation}
Furthermore, up to equivalence of graphons,
\begin{equation}\label{eq:convexlimits}
  \maxlimits{\CI} = \{ W^4_a : a \in [0, 1] \}. 
\end{equation}
\end{theorem}

\begin{proof}
It is shown
in~\cite{KK98} that $\CC(4,4) \subseteq \CI$ and,
more generally, that every graph that
admits a planar clique covering is a convex intersection graph.  Hence, we
have $\col(\CI) \ge 4$, and 
since $\CI\subseteq\SG$, we have $\col(\CI)\le\col(\SG)=4$.
This proves \eqref{eq:convexcoloringnumber} and \eqref{eq:convexspeed}
follows by \eqref{eq:speed}. Finally, \eqref{eq:convexlimits} follows by 
the same argument as in the proof of
\refT{th:stringlimits}.
\end{proof}

\section{Proof of Theorem~\ref{th:outerstringlimits}}\label{se:outerstringproof}

\begin{lemma}\label{le:outerstringforb}
The graphs $G_4$,~$B_4$, and~$H_4$ are not outer-string graphs.
\end{lemma}

\begin{proof}
Suppose to the contrary that $G_4$,~$B_4$, or~$H_4$ has an outer-string representation.  If necessary, we may assume (possibly by extending the outer-string representation to a slightly larger disk)  that all curves meet the disk at distinct points.
Then we may extend this representation to a string representation of
$G_5$,~$B_5$, or~$H_5$ by adding curves 
$A_{\{5\}}$, $A_{\set{1,5}}$, \ldots,~$A_{\set{4,5}}$ outside of the disk
with the appropriate intersection pattern. 
(See also Theorem~\ref{th:classes}.)
However, this contradicts Lemma~\ref{le:stringforb}, and the claimed result follows.
\end{proof}

We will also need the following result of Pach and T\'oth~\cite[Lemma 3.1]{PT06}.

\begin{lemma}[\cite{PT06}]\label{le:Ginclusion}
If $r \geq 1$, then $G_{r+1} \in \CC(r + 1, s)$ for $s = 0$, \dots, $r + 1$.
In particular, if $\PP$ is a hereditary property and $\PP \subseteq
\Forb^*(\set{G_{r+1}})$, then $\col(\PP) \leq r$.
\qed
\end{lemma}

\begin{lemma}\label{le:osinclusion}
Every graph in the class~$\mathcal{C}(3, 3)$ is an outer-string graph.
Moreover, every graph that admits an outerplanar clique covering is an
outer-string graph. 
\end{lemma}

\begin{proof}
As stated above, the corresponding results for string graphs were proved in
\cite[Theorem 2.3]{KGK} and in
\cite[Example 2.4]{HJS}.  We modify the construction given in the latter.

Let $G$ be a graph whose vertex set can be covered by three cliques
$V_1$,~$V_2$, and~$V_3$.  Place distinct points $v_1$,~$v_2$, and~$v_3$ on
the boundary of a disk $D$.  For each $x \in V_i$ and~$y \in V_j$ with $i \neq
j$, add a curve~$\tilde{A}_{xy}\subset D$ 
between $v_i$ and $v_j$.  We may place the 
curves in such a way that different curves do not meet apart from their
endpoints.  For each $x \in V_i$ and~$y \in V_j$ with $i \neq j$, choose a
point~$a_{xy} \in \tilde{A}_{xy}$ that is different from both $v_i$ and
$v_j$.  Let $\tilde{A}_{xy}^*$ denote the portion of~$\tilde{A}_{xy}$
between $v_i$ and $a_{xy}$ and let $\tilde{A}_{yx}^*$ denote the portion
of~$\tilde{A}_{xy}$ between $v_j$ and $a_{xy}$.   
Let, for $x\in V_i$,
\[
	A_x = \Biggl(\bigcup_{\substack{xy \in E(G),\\y \notin V_i}} \tilde{A}_{xy}^* \Biggr).
\]
It is easy to see that the collection of curves~$\{A_v : v \in V(G)\}$ gives
an outer-string representation of~$G$.
(We can regard each $A_x$ as a curve starting and ending at $v_i$,
traversing each $\tilde{A}_{xy}^*$ in both directions, with these parts in
arbitrary order.) 

In general, if $G$ admits an outerplanar clique covering with $\covering = \{V_1, \dots, V_k\}$,
we may place $k$ distinct points on the boundary of the disk and repeat the construction described above to define an outer-string representation of~$G$.
\end{proof}

\begin{lemma}\label{le:outerstringcoloring}
We have
\begin{equation*}
	\col(\OSG) = 3
\end{equation*}
and
\begin{equation}\label{eq:outerstringspeed}
	\abs{\OSG_n} = 2^{\left(\frac{2}{3} + o(1)\right)\binom{n}{2}}.
\end{equation}
\end{lemma}

\begin{proof}
The fact that $\col(\OSG) = 3$ is immediate from Lemmas \ref{le:outerstringforb},~\ref{le:Ginclusion}, and~\ref{le:osinclusion}.  This and \eqref{eq:speed} imply~\eqref{eq:outerstringspeed}.
\end{proof}

\begin{remark}\label{re:proper}
Because $G_4 \in \CC(4,4) \subseteq \SG$, $G_4$ is an example of a string
graph that is not an outer-string graph.  There are many others:
\eqref{eq:stringspeed} and \eqref{eq:outerstringspeed}
imply that almost every string graph is not an outer-string graph. 
\end{remark}

%

\begin{proof}[Proof of Theorem \ref{th:outerstringlimits}.]
One inclusion is immediate from Lemma~\ref{le:outerstringforb}, Lemma~\ref{le:outerstringcoloring}, and Theorem~\ref{th:general}.  For the other, let $a \in [0,1]$.  As in the proof of \refT{th:stringlimits}, one can define a sequence of graphs~$G_n=G(n,W^3_a)$ such that for each~$n$, with probability~one, $G_n$ admits a clique covering~$\covering$ such that $G_{\covering} = K_4 - e$, and such that $G_n \to W^3_a$ almost surely as $n \to \infty$.
Because $K_4 - e$ is outerplanar, it follows from Lemma~\ref{le:osinclusion} that each $G_n$ is an outer-string graph. Theorem~\ref{th:maxlimits} then implies that $W^3_a \in \limitsx{\OSG} \cap R_3 = \maxlimits{OSG}$, as claimed.
\end{proof}


\section{Proof of Theorem~\ref{th:incomplimits}}\label{se:incompproof}

\begin{lemma}\label{le:incompforb}
The graphs $G_3$ and~$B_3$ are not incomparability graphs.
\end{lemma}

\begin{proof}
Observe that $B_3 = C_6$ and that $G_3$ is 
the complement of a triangle with a pendant edge attached to each vertex.  It is a well-known result of Gallai~\cite{Gallai} (see also~\cite{Trotter}) that neither of these graphs is an incomparability graph.

Alternatively, suppose that $G_3$ or $B_3$ had an intersection representation consisting of strings between two parallel segments.  By arguing as in the proof of Lemma~\ref{le:outerstringforb}, we could extend this representation to an outer-string representation of $G_4$ or $B_4$, a contradiction.
\end{proof}

\begin{lemma}\label{le:incompcol}
We have
\begin{equation*}
	\col(\incomp) = 2
\end{equation*}
and
\begin{equation}\label{eq:incompspeed}
	\abs{\incomp_n} = 2^{\left(\frac{1}{2} + o(1)\right)\binom{n}{2}}.
\end{equation}
\end{lemma}

\begin{remark}\label{re:incompspeed}
The asymptotic speed of the class of incomparability graphs (equation~\eqref{eq:incompspeed}) was first determined by Kleitman and Rothschild~\cite{KR70,KR75}.
\end{remark}

\begin{proof}[Proof of Lemma \ref{le:incompcol}.]
It is easy to see that every bipartite graph is a comparability graph, which is equivalent to the statement that $\CC(2, 2) \subseteq \incomp$.  It follows that $\col(\incomp) \geq 2$.  Lemmas \ref{le:incompforb} and~\ref{le:Ginclusion} then imply that $\col(\incomp) = 2$.  Finally, \eqref{eq:incompspeed} follows from~\eqref{eq:speed}.
\end{proof}

In the proof of Theorem~\ref{th:incomplimits}, instead of directly proving our results about incomparability graphs, we will find it convenient to prove the corresponding results for comparability graphs.  We let $\comp$ denote the class of comparability graphs.

\begin{proof}[Proof of Theorem \ref{th:incomplimits}.]
One inclusion is immediate from Lemma~\ref{le:incompforb}, Lemma~\ref{le:incompcol}, and Theorem~\ref{th:general}.  In order to show that
\[
\maxlimits{\incomp} \supseteq \bigl\{ W^2_a \, : \, a \in [0, 1] \bigr\},
\]
it is enough to prove that $1 - W^2_a \in \maxlimits{\comp}$ for every $a \in \oi$.

For $a \in \set{0, 1}$, the result follows from the fact that every
bipartite graph is a comparability graph.  Hence, we may suppose that $a \in
(0, 1)$. 
For each~$n$, we let $G_n = G(n, 1-W^2_a)$. 
Let $A_n = \set{i : X_i \in [0, a/2)}$, $C_n = \set{i : X_i \in [a/2,1/2)}$,
and $B_n = \set{i : X_i \in [1/2, 1)}$. 
Observe that with probability~one, each of $A_n$,~$B_n$, and~$C_n$ is an
incomparable set, and $A_n\cup C_n$ is a complete bipartite graph.  
Hence, if we orient edges from $A_n$ to~$B_n$, from $A_n$ to $C_n$ 
and from $B_n$ to $C_n$,
then if $ab$ and
$bc$ are directed eges, then $ac$ is almost surely a directed edge.
Thus, each $G_n$ is almost surely a comparability graph. 
Because $G_n \to 1 - W^2_a$ almost surely and $\col(\comp) = \col(\incomp) =
2$, it follows from Theorem~\ref{th:maxlimits} that $1 - W^2_a \in
\limitsx{\comp} \cap R_2 = \maxlimits{\comp}$, as claimed. 
\end{proof}

\section{Random string graphs}\label{se:random}

Let $\QQ$ be any graph class.  The next result says that if the maximum
entropy of an element of~$\limitsx{\QQ}$ determines the speed of~$\QQ$, then
the maximum-entropy elements of~$\limitsx{\QQ}$ also determine the 
asymptotic structure
of a typical element of~$\QQ$.  This was proved in~\cite[Theorem 1.6]{HJS}
in the special case when $\maxlimits{\QQ}$ consists of a single element.
Since the proof of the general case is nearly identical to the argument
given in~\cite{HJS}, we omit it. 

Given $\delta > 0$, let $A_{\delta}(\QQ) = \{\Gamma \in \limitsx{\graphs} :
\deltacut(\Gamma, \maxlimits{\QQ}) < \delta\}$, 
where $\deltacut$ is the standard metric on the space of graph limits, see
\eg{}~\cite{LovaszBook}.

\begin{theorem}\label{th:convinprob}
Let $\QQ$ be a class of graphs and suppose that
\begin{equation}\label{eq:speedlimit}
	\lim_{n \to \infty} \dfrac{\log_2\abs{\QQ_n}}{\binom{n}{2}} = \max_{\Gamma \in \limitsx{\QQ}} \Ent(\Gamma).
\end{equation}
For each~$n$, let $G_n$ be a uniformly random unlabeled element of~$\QQ_n$.  For any $\delta > 0$,
\[
	\lim_{n \to \infty} \Pr\bigl(G_n \in A_{\delta}(\QQ)\bigr) = 1.
\]
Moreover, the same conclusion holds if we let each $G_n$ be a uniformly
random labeled element of~$\QQ_n$. \qed
\end{theorem}

\begin{remark}\label{re:heredspeed}
By Theorem~\ref{th:maxlimits}, equation~\eqref{eq:speedlimit} holds whenever
$\QQ$ is a hereditary property. 
\end{remark}

It is immediate from Theorem~\ref{th:convinprob} and Theorems
\ref{th:stringlimits} and~\ref{th:outerstringlimits}, respectively, that if
$G_n$ is a uniformly random (unlabeled or labeled)
string graph, then $G_n$ converges in probability as $n\to\infty$ to the
set~$\{W^4_a : a \in [0, 1]\}$, and similarly that a sequence of uniformly
random outer-string graphs converges in probability to the set~$\{W^3_a : a \in [0, 1]\}$.
In principle, it is possible that the sequence $G_n$ converges in
distribution to a limit that is a non-degenerate random graph limit in
$\set{W^4_a}$, or that the distributions oscillate and do not converge at
all, but it seems very probable that there exists some $a\in[0,1/2]$ such
that
$G_n \pto W^4_a$.
However, assuming that this really holds, what is the limiting $a$?
We believe that $a=1/2$, and thus that the following stronger results should also
hold. 

\begin{conjecture}\label{conj:stringconv}
If, for each~$n$, $G_n$ is a uniformly random unlabeled element of~$\SG_n$, then $G_n \pto W^4_{1/2}$.  Moreover, the same conclusion holds if each $G_n$ is a uniformly random labeled element of~$\SG^L_n$.
\end{conjecture}

\begin{conjecture}\label{conj:outerstringconv}
If, for each~$n$, $G_n$ is a uniformly random unlabeled element of~$\OSG_n$, then $G_n \pto W^3_{1/2}$.  Moreover, the same conclusion holds if each $G_n$ is a uniformly random labeled element of~$\OSG^L_n$.
\end{conjecture}

We believe that these results should hold for two reasons.  First,
in the proof of \refT{th:stringlimits}, the value $a=1/2$ gives
the largest number of partitions of the vertex set. 
(And similarly for outer-string graphs, see the proof of
\refT{th:outerstringlimits}.) 
Second, the corresponding results for incomparability graphs (at least in the labeled case) and for two-clique graphs are known to hold.  The result for labeled incomparability graphs follows from the corresponding result for partial orders, which was proved in~\cite{KR75} (see also~\cite{BPS96}). We believe that the same result should hold for unlabeled incomparability graphs; it is a folklore result that almost every partial order has trivial automorphism group, and we believe that the same is true of incomparability graphs, but we do not know a proof of this statement.  The result for two-clique graphs is trivial; for a sketch of a nearly identical argument, see~\cite[Example 7.9]{Jan13b}.

\begin{theorem}\label{th:incompconv}
If, for each~$n$, $G_n$ is a uniformly random labeled element
of~$\incomp^L_n$, then $G_n \pto W^2_{1/2}$. \qed
\end{theorem}

\begin{theorem}\label{th:2cliqueconv}
If, for each~$n$, $G_n$ is a uniformly random unlabeled element
of~$\mathcal{TCG}_n$, then $G_n \pto W^1_{1/2}$.  If each $G_n$ is a
uniformly random labeled element of~$\mathcal{TCG}^L_n$, then $G_n$ converges in probability to the random graphon~$W^1_T$, where $T
\sim U(0, 1/2)$. \qed
\end{theorem} 

It is easy to see that the limiting distribution of a uniformly random
string graph, if it exists, determines the limiting distribution of the edge
density of a uniformly random string graph (and similarly for outer-string
graphs).  It also determines the limiting distribution of the proportion of
vertices of degree approximately~$cn$, where $c \in \oi$ is a constant;
it is convenient to state this using the distribution of the degree of a
uniformly random vertex. (See e.g.\ \cite[Section 4]{DHJ}.)
Thus, Conjectures \ref{conj:stringconv} and~\ref{conj:outerstringconv}
imply the following conjectures, including
the somewhat suprising statement that the degree distribution
of a typical string graph (and of a typical outer-string graph) is bimodal. 

Given a graph~$G$ on $n$ vertices, let $X_n$ denote the degree of a uniformly random vertex of~$G$.

\begin{conjecture}\label{conj:sgedgesdegrees}
If, for each~$n$, $G_n$ is a uniformly random unlabeled element of~$\SG_n$, then
\begin{equation}\label{eq:sgedges}
\dfrac{e(G_n)}{\binom{n}{2}} \pto \dfrac{19}{32}.
\end{equation}
Furthermore, $X_n / n$ converges in distribution to a random variable~$Y$ such that $y \in \set{1/2, 5/8}$ almost surely and such that
\[
\Pr(Y = 1/2) = \dfrac{1}{4} \qquad \text{and} \qquad \Pr(Y = 5/8) = \dfrac{3}{4}.
\]
Moreover, the same conclusions hold if each $G_n$ is a uniformly random labeled element of~$\SG^L_n$.
\end{conjecture}

\begin{conjecture}\label{conj:osgedgesdegrees}
If, for each~$n$, $G_n$ is a uniformly random unlabeled element of~$\OSG_n$, then
\begin{equation}\label{eq:osgedges}
\dfrac{e(G_n)}{\binom{n}{2}} \pto \dfrac{11}{18}.
\end{equation}
Furthermore, $X_n / n$ converges in distribution to a random variable~$Y$ such that $y \in \set{1/2, 2/3}$ almost surely and such that
\[
\Pr(Y = 1/2) = \dfrac{1}{3} \qquad \text{and} \qquad \Pr(Y = 2/3) = \dfrac{2}{3}.
\]
Moreover, the same conclusions hold if each $G_n$ is a uniformly random labeled element of~$\OSG^L_n$.
\end{conjecture}

Let us also remark that because $W^4_{1/2}$ has minimum edge density in the
set~$\{W^4_a : a \in \oi\}$ (cf.~\eqref{eq:Wdensity}), the
assertion~\eqref{eq:sgedges} is
actually equivalent to Conjecture \ref{conj:stringconv};
moreover, \eqref{eq:sgedges} is 
also equivalent to saying that the edge
density of a random string graph converges in expectation to~$19/32$.
Similarly, \eqref{eq:osgedges} is equivalent to the analogous statements about
outer-string graphs.

\section{Proof of Lemma \ref{le:magic}}\label{se:removal}

We find it convenient to make two new definitions, closely related to
concepts in Section \ref{se:tools} but from a slightly
different point of view.
Recall \eqref{psiqfw}--\eqref{tindfw}.

\begin{definition}\label{D1}
Let $W$ be a graphon and $G$ a graph with vertex set $[n]$.
  \begin{romenumerate}[-15pt]
  \item 
$W$ is $G$-free if $\tind(G;W)=0$.
By \eqref{tindfw}, this is equivalent to
\begin{equation}
\psiqfw(\xxm)=0  
\qquad\text{for \aex{} $\xxm\in\oi$}.
\end{equation}
\item 
$W$ is completely $G$-free if 
\begin{equation}
\psiqfw(\xxm)=0  
\qquad\text{for {all} $\xxm\in\oi$}.
\end{equation}
In other words, $G$ is not $W$-constructible.
  \end{romenumerate}
\end{definition}
We emphasize that in (ii), the sequence $\xxm$ is completely arbitrary, and
may contain repetitions. In particular, we require $\psiqfw(x,\dots,x)=0$
for every $x\in\oi$, which by \eqref{psiqfw} implies $W(x,x)\in\setoi$ for
every $x\in\oi$.

As said in the introduction, 
by \cite[Theorem 3.3]{Jan13b}, 
if $\PP$ is hereditary, then a graphon $W\in\hcP$
if and only if $\tind(G;W)=0$ for every graph $G\notin\cP$; in other words,
if and only if $W$ is \ffree{} for every $G\notin\cP$.
The following theorem says that this can be strengthened to completely
\ffree.

\begin{theorem}\label{Tfree}
  Let $\cP$ be a hereditary graph property and $\gG$ a graph limit.
Then the following are equivalent.
\begin{romenumerate}
\item 
$\gG\in\hcP$.
\item 
Every graphon representing $\gG$ is \ffree{} for every $G\notin\cP$.
\item 
There exists a graphon representing $\gG$ that is \ffree{} 
for every $G\notin\cP$.
\item 
There exists a graphon representing $\gG$ that is completely \ffree{}
for every $G\notin\cP$.
\end{romenumerate}

Moreover, suppose that $\oi=\bigcup_{k=1}^r A_k$ is a partition into
finitely many measurable sets with positive measures, 
and that we are given a subset 
$\QE\subseteq [r]\times [r]$ and real numbers $a_{kl}\in \oi$ for $(k,l)\in
\QE$. 
Suppose further that 
the properties above hold 
and that there is a graphon $W$ representing $\gG$ such that 
\begin{equation}\label{waij}
  W(x,y)=a_{kl}
\quad\text{for all } (k,l)\in \QE
\text{ and all }(x,y)\in A_k\times A_l.
\end{equation}
Then there exists a graphon $W$ representing $\gG$ that is completely
\ffree{} for every $G\notin\cP$ and is such that \eqref{waij} holds. 
\end{theorem}

Obviously, it suffices to assume that \eqref{waij} holds for \aex{}
$(x,y)\in A_k\times A_l$, since we can begin by redefining $W$ on a set of
measure 0 so that \eqref{waij} holds for all $(x,y)\in A_k\times A_l$.

\begin{proof}
  The equivalences (i)$\iff$(ii)$\iff$(iii) follow, as said above, from 
\cite[Theorem 3.3]{Jan13b} and Definition \ref{D1}.
The implication (iv)$\implies$(iii) is trivial. 
It thus suffices to prove (iii)$\implies$(iv), and the final statement.

For this, suppose that $\gG$ is represented by a graphon $W$ that is
\ffree{}  for every 
$G\notin\cP$. 
Let $\cI$ be the set of all pairs $ij=(i,j)$ of distinct positive
integers, and 
for every labeled graph $G$, define $M_G\subset\oi^\cI$ 
by
\begin{multline}\label{mf}
M_G=
\Bigset{(w_{ij})\in\oi^{\cI}:w_{ij}=w_{ji}
\text{ for all $i,j\in\bbN$ ($i\neq j$)}
\\
\text{and }
\prod_{ij\in E(G)} w\qij
\prod_{ij\not\in E(G)}(1- w\qij)=0.
}
\end{multline}
Furthermore, define
\begin{equation}\label{m}
M=\bigcap_{G\notin\cP}M_G.  
\end{equation}
If $G\notin\cP$, then $W$ is assumed to be \ffree, which by \eqref{psiqfw}
and \eqref{mf} means that the infinite vector 
$(W(x_i,x_j))\qij\in M_G$ for
\aex{} sequence $x_1,x_2,\dots\in\oi$ (with the product measure).
Since the number of labeled graphs is countable, it follows that
$(W(x_i,x_j))\qij\in M$ for
\aex{} $x_1,x_2,\dots\in\oi$. 

By the general removal lemma of Petrov \cite[Theorem 1(2)]{petrov}
(taking $K=X=\oi$  and $k=2$ there),
there exists a graphon $W'$ such that $W'(x,y)=W(x,y)$ \aex{} 
(so $W'$ also represents $\gG$)
and moreover
$(W'(x_i,x_j))\qij\in M$ for
all $x_1,x_2,\dots\in\oi$.
By \eqref{mf}--\eqref{m} and \eqref{psiqfw}, it follows that $W'$ is
completely \ffree{} for every $G\notin\cP$, which proves (iv).

For the final statement, we may, by a suitable measure-preserving bijection
$\oi\to[0,1)$, assume that the graphons are defined on $[0,1)$ and 
that the sets~$A_k$ are half-open intervals $[a_k,b_k)$.
We look into the proof by Petrov  \cite{petrov}; there is defined a set
$Y\subseteq[0,1)^2$ consisting of all pairs $(x_1,x_2)\in[0,1)^2$ 
such that if $U$ is any open
set containing $f(x_1,x_2)$, 
$I_{x,m}=\set{y:\floor{my}=\floor{mx}}$  
(an interval of length~$1/m$ containing $x$),
and $\mu$ is Lebesgue measure in the plane, 
then
\begin{equation}
\lim_{m\to\infty}
m^2 \mu\bigset{(y_1,y_2)\in I_{x_1,m}\times I_{x_2,m}:f(y_1,y_2)\in U}
= 1.
\end{equation}
The construction of $W'$ is such that $W'=W$ on $Y$. 
It follows immediately from~\eqref{waij} that if $(k,l)\in \QE$, then
$A_k^\circ\times A_l^\circ\subseteq Y$,
where $A_k^\circ=(a_k,b_k)$,
and thus 
$W'=W=a_{kl}$ on $A_k^\circ\times A_l^\circ$.
Finally, pick $a_k'\in(a_k,b_k)$, and define $\phi:[0,1)\to[0,1)$
by $\phi(a_k)=a_k'$ and $\phi(x)=x$ for $x\notin\set{a_1,\dots,a_r}$.
Then $W''(x,y):=W'(\phi(x),\phi(y))$ equals $W$ \aex{}, is completely
$G$-free for every $G\notin\PP$, and satisfies \eqref{waij}.
\end{proof}

\begin{remark} 
  We derived the last statement in \refT{Tfree} by using the proof by
Petrov  \cite{petrov}. It is also possible to use only the statement together
  with a more complicated construction, which we sketch here:
Consider the finite set $B=\bigset{\set{i,j}:i,j\in[r]}$ of unordered pairs
or singletons
and the function $g:\oi^2\to B$ defined by $g=\set{i,j}$ on $A_i\times A_j$.
(We consider unordered pairs $\set{i,j}$ since we need $g$ to be symmetric.)
We apply Petrov's theorem to the function $\bigpar{W(x,y),g(x,y)}$
mapping $\oi^2$ into the compact set $K=\oi\times B$, and define 
$M\subset K^\cI=\oi^\cI\times B^\cI$  by the conditions in \eqref{mf} for
all $G\notin\cP$ together with conditions corresponding to 
\begin{romenumerate}
\item 
$g(x,y) \cap g(x,z)\neq\emptyset$;
\item 
$g(x,y)=\set{i,j}$ and $g(y,z)=\set{i,k}$ with $j\neq k$ $\implies$
$g(x,z)=\set{j,k}$;
\item 
$g(x,y)=\set{k,l} \text{ and } (k,l)\in\QE\implies W(x,y)=a_{kl}$.
\end{romenumerate}
Petrov's theorem yields $W'$ and $g'$ satisfying corresponding conditions
everywhere, and it can be seen that 
(at least if $r\ge3$, as we may assume by splitting some~$A_k$)
there exists a partition $\oi=\bigcup_k
A_k'$ such that $g'(x,y)=\set{k,l}$ and thus $W'(x,y)=a_{kl}$ on $A_k'\times
A_l'$, and furthermore $A_k\setminus A_k'$ has measure 0 for each
$k$. 
Finally we redefine $W'(x,y)$ when $x$ or $y\in A_k\setminus A_k'$ for some 
$k$. We omit the details. This argument can also be used when the 
condition $W(x,y)=a_{kl}$ in \eqref{waij} is replaced by $W(x,y)\in A^*_{kl}$
for some compact sets $A^*_{kl}$.
\end{remark}

\begin{proof}[Proof of Lemma \ref{le:magic}]
By \refT{Tfree}, there exists a graphon $W$ representing $\gG$ that is
completely \ffree{} for every $G\notin\cP$. 
If $G$ is $W$-constructible, then
$W$ is not completely $G$-free, so  
$G\in\cP$.

Moreover, if\/ $\Gamma \in \maxlimits{\PP}$  and  $r = \col(\PP)$, 
then \refT{th:maxlimits} shows that there exists 
a representing graphon $W \in R_r$.  If $r = \infty$, i.e., $\PP = \graphs$, then it is easy to see that the constant graphon~$W = 1/2$ possesses the required properties.
Otherwise, we can apply the last statement in \refT{Tfree} with
$A_k=I_k=[(k-1)/r,k/r)$, $\QE=\set{(k,\ell):k\neq\ell}$ and
$a_{k\ell}=1/2$ for $(k,\ell)\in\QE$.
This shows that there exists a representing graphon $W'$ that is completely
\ffree{} for every $G\notin \cP$  with $W'=1/2$ 
on $\bigcup_{i\neq j} I_i\times I_j$.
Furthermore, $W'=W\in\setoi$ \aex{} on each $I_i\times I_i$.
Define $W''$ as the modification of $W'$ obtained by letting $W''(x,y)=0$
if $x,y\in I_i$ for some $i$, and $0<W'(x,y)<1$, and otherwise $W''=W'$.
Then $W''=W'$ a.e., so $W''$ represents $\gG$,
$W''\in R_r$ and
$W''$ is still completely \ffree{} for every $G\notin\cP$. This completes
the proof.
\end{proof}

\section*{Acknowledgments}\label{se:ack}
We would like to thank J\'anos Pach 
and Graham Brightwell for helpful discussions. 

\bibliographystyle{plain}
\bibliography{graphlimitsbibnew}

\appendix
\section{String Representations and Arcwise-Connected Sets}\label{se:arcwiserep}

We have defined string graphs as graphs having an intersection
representation using curves in the plane.
It is well-known that the string graphs can also be defined in terms of intersection representations of other types of planar sets.
We collect some of these equivalences, and for completeness give a 
proof. 

\begin{lemma}\label{le:strings}
The following families of subsets of the plane $\mathbb R^2$
define the same class of intersection graphs, i.e., the string graphs:
  \begin{romenumerate}
  \item \label{a:curves}
curves 
(continuous images of $\oi$),
  \item \label{a:simple}
simple curves 
(homeomorphic images of $\oi$),
  \item \label{a:closed}
simple closed curves 
(homeomorphic images of $S^1$),
  \item \label{a:plinear}
piecewise-linear simple closed curves,
  \item \label{a:pathwise}
pathwise-connected sets,
  \item \label{a:arcwise}
arcwise-connected sets,
  \item \label{a:open}
open connected sets,
  \item \label{a:simply}
open simply connected sets.
  \end{romenumerate}
Furthermore, in \ref{a:curves}--\ref{a:plinear} we may assume that the
intersection representation is such that 
such that every pair of strings intersects finitely many times,
such that no three curves cross at a single point,
and such that if two strings cross, then they cross properly.
\end{lemma}

Recall that a topological space is \emph{pathwise-connected} if any two
points in it can be connected by a curve, and \emph{arcwise-connected} if 
any two points can be connected by a simple curve. These notions are
actually equivalent for Hausdorff spaces (and thus for subsets of the
plane), see \cite[Problem 6.3.12]{Engelking}.
Any open connected set in the plane is pathwise-connected.

\begin{proof}
The implications
\ref{a:plinear}$\implies$\ref{a:closed}$\implies$\ref{a:curves}
$\implies$\ref{a:pathwise},
\ref{a:simple}$\implies$\ref{a:curves},
\ref{a:simple}$\implies$\ref{a:arcwise}$\implies$\ref{a:pathwise},
and 
\ref{a:simply}$\implies$\ref{a:open}$\implies$\ref{a:pathwise}
are trivial. To complete the proof it thus suffices to show that
\ref{a:pathwise}$\implies$\ref{a:simple},\ref{a:plinear} and
\ref{a:plinear}$\implies$\ref{a:simply}.

\ref{a:pathwise}$\implies$%
\ref{a:curves},\ref{a:simple},%
\ref{a:closed},\ref{a:plinear}. 
Suppose that $G$ has an intersection representation consisting of pathwise-connected sets~$\{A_v : v \in V(G)\}$.  
For convenience, we may assume that $G$ has no isolated vertices, since these
can be added at the end (or by trivial modifications in the argument below).
For each~$i$, choose a point~$p_i\in A_i$.  
For each~$j\neq i$ such that $A_i \cap A_j \neq \emptyset$, choose a
point~$p_{ij} \in A_i \cap A_j$.  
By hypothesis, for each~$i$ and for each~$j\neq i$ such that $A_i \cap
A_j \neq \emptyset$, we may choose a curve~$C_i^j \subseteq
A_i$ with endpoints $p_i$ and~$p_{ij}$.  
(Note that this notation is
\emph{not} symmetric: in general, $C_i^j \neq C_j^i$.)  
Define $A_i' = \bigcup_j \Cij$. Then $A_i'\subseteq A_i$, each $A_i'$ is
pathwise-connected and compact, and $\set{A_i'}$ yields an intersection
representation of $G$. Hence, by replacing $A_i$ by $A_i'$, we may assume
that each set $A_i$ is compact, which
implies that there exists~$\varepsilon > 0$ such that if $A_i \cap A_j =
\emptyset$, then $d(A_i, A_j) \geq 3\varepsilon$. 
Furthermore, we may at each point $p_{ij}$ add a small line segment (of length less than~$\varepsilon$ and with $p_{ij}$ as one endpoint) to both $A_i$ and
$A_j$ without creating any new 
intersecting pairs of sets; this guarantees that any pair of sets that
intersects will intersect in infinitely many points.

We now start again with these modified sets $A_i$ and choose new
points~$p_i\in A_i$ and 
$p_{ij} \in A_i \cap A_j$, 
for each~$j\neq i$ such that $A_i \cap A_j \neq \emptyset$.
We choose these points such
that $p_{ij} = p_{ji}$ when these points are defined.
However, we require that $p_{ij}
\neq p_{kl}$ otherwise; that 
$p_i \neq p_{jk}$ for all $i$,~$j$, and~$k$;
and that $p_i\neq p_j$ for $i\neq j$. (Note that this is possible by the
modifications above.)
Once again, choose a curve~$C_i^j \subseteq A_i$ with endpoints $p_i$ and~$p_{ij}$.  
Since $\Cij\subseteq A_i$,
$d(C_i, C_j) \geq \varepsilon$ whenever $A_i\cap A_j=\emptyset$, but
$C_i\cap C_j\supseteq\set{p_{ij}}$ if $A_i\cap A_j\neq \emptyset$.

For each~$i$ and for each~$j\neq i$ such that $A_i \cap A_j \neq \emptyset$, let
$\tilde{C}_i^j$ be a piecewise-linear approximation of~$C_i^j$ 
(with the same endpoints $p_i$ and $p_{ij}$, and at least two linear segments)
to within
distance~$\varepsilon/3$.  (That is, for all~$x \in \tilde{C}_i^j$, there
exists~$y \in C_i^j$ such that $d(x, y) < \varepsilon/3$.)  Furthermore, if
necessary, we may randomly perturb the endpoints of the segments in each
$\tilde{C}_i^j$ (except $p_i$ and $p_{ij}$)
such that every such endpoint is contained in only~one curve~$\tilde{C}_i^j$,
$p_i$ is not contained in any $\tC_k^\ell$ with $k\neq i$,
$p_{ij}$ is contained only in $\tC_i^j$ and $\tC_j^i$,
no point outside of~$\set{p_i}$ lies on three curves~$\tilde{C}_k^\ell$,
and 
all segments have pairwise different directions.
Note that our choice of~$\varepsilon$ means that if $A_i \cap A_j =
\emptyset$, then for any $k$ and~$\ell$, we have $\tilde{C}_i^k \cap
\tilde{C}_j^{\ell} = \emptyset$. 

For each~$i$ and for each~$j$ such that $\tilde{C}_i^j$ is defined, let $N =
N_i^j$ denote the number of segments in the curve~$\tilde{C}_i^j$.  Let
$s_i^j(m)$ denote the $m$th segment in order of increasing distance from
$p_i$.  Furthermore, let $p_i = x_i^j(0)$, \dots, $x_i^j(N) = p_{ij}$ denote
the sequence of endpoints of segments in the curve~$\tilde{C}_i^j$, so that
$x_i^j(m-1)$ and $x_i^j(m)$ are the endpoints of~$s_i^j(m)$.

For each~$i$, for each~$j$ such that $\tilde{C}_i^j$ is defined, and for $m
\in \{0, \dots, N\}$, let $u_i^j(m)$ be a segment of length~$\eta$ centered
at~$x_i^j(m)$, where $\eta<\varepsilon/6$ is a small number.  
If $m \in \{1, \ldots, N-1\}$, then we define $\uij(m)$ to
lie on the angle-bisector of $\sij(m)$~and~$\sij(m+1)$.
We define $\uij(0)$ to be perpendicular to $\sij(1)$ and $\uij(N)$ to be perpendicular
to $\sij(N)$.  For each $m$, let $\xij(m)^+$ be the right-hand endpoint
of~$\uij(m)$ as $\tilde{C}_i^j$ is oriented from $p_i$ to~$p_{ij}$, and let
$\xij(m)^-$ denote the left-hand endpoint.  Let $\Tij(m)$ denote the
trapezoid with vertices $\xij(m-1)^+$, $\xij(m-1)^-$, $\xij(m)^-$,
and~$\xij(m)^+$. 

For each~$i$ and for each~$j$ such that $\tilde{C}_i^j$ is defined, let
$\Bij = \cup_{m=1}^N \Tij(m)$.  Note that the choices of~$\varepsilon$ and
$\eta$ still
ensure that if $A_i \cap A_j = \emptyset$, then for any $k$ and~$\ell$, we
have $B_i^k \cap B_j^{\ell} = \emptyset$.  
If $\eta$ is small enough, then 
the boundary $D_i^j=\partial B_i^j$ is a piecewise-linear simple closed curve,
and no point lies on three of these curves, except that $p_i$ lies on $D_i^j$
for each~$j$ such that this curve is defined;
making a small random perturbation of the segments through $p_i$, we may 
assume that no point at all lies on three of these curves, while
$D_i^j$ still intersects $ D_i^k$ 
for any $j$ and~$k$ such that these curves are defined. Note also that
$p_{ij}\in D_i^j$. 
Furthermore, if necessary after another random perturbation of the segments
through $p_{ij}$,
all segments have pairwise different directions and since each string
consists of finitely many segments, all pairs of curves $D_i^j$ have
finitely many 
intersection points.  This and the fact that every endpoint is contained in
only~one such curve 
(if $\eta$ is small enough)
means that if two curves cross, then they cross properly.

For each~$i$, let $D_i = \bigcup_j D_i^j$, where once again the union is
taken over all~$j$ for which $B_i^j$ is 
defined. Once again, $D_i \cap D_j \neq \emptyset$ if and only if $A_i \cap
A_j \neq \emptyset$.  Each $D_i$ is a union of  polygonal curves, and
$D_i$ is connected. 
Furthermore,
viewing $D_i$ as a graph, each vertex has degree 2 or 4, so the graph is
Eulerian.  Moreover, viewing an Eulerian circuit as a directed curve, it is easy to see that the graph has an Eulerian circuit that does not properly cross itself.
Hence, making a small
modification at each point where $D_i$ intersects itself, we may replace
$D_i$ by a piecewise-linear simple closed curve $\gamma_i$, without creating
or destroying any other intersections.

This yields an intersection representation~$\set{\gamma_i}$ of $G$
consisting of piecewise-linear simple closed curves (as in \ref{a:plinear}),
which furthermore satisfies the properties in the final part of the
statement.
By choosing a point $q_i$ on each $\gamma_i$, not on any other of the
curves, and deleting a small open interval about $q_i$, we can replace the
closed curves $\gamma_i$ by simple curves $\gamma_i'$ 
as in \ref{a:simple}.
This completes the proof that
\ref{a:pathwise}$\implies$\ref{a:curves},\ref{a:simple},\ref{a:closed},\ref{a:plinear}.

\ref{a:plinear}$\implies$\ref{a:simply}:
Let $\set{A_i}$ be an intersection representation of~$G$ consisting of
piecewise-linear simple curves. Extend the first and last segments a tiny
bit $\eta$ past the original endpoints.
Then, similarly to the proof above, define a region~$B_i\supset A_i$ as the
union 
of small trapezoids of width at most~$\eta$
surrounding the segments in $A_i$.
Then, if $\eta$ is small enough, the interiors $B_i^\circ$ form an
intersection representation of~$G$
consisting of open simply connected sets.
\end{proof}

A similar result holds for outer-string graphs.
We give one version, leaving further variations to the reader.

\begin{lemma}
Let $D$ be a closed disk in the plane $\mathbb R^2$.
The following families of subsets of $D$
define the same class of intersection graphs, i.e., the outer-string graphs:
  \begin{romenumerate}
  \item \label{o:curves}
curves with one endpoint in $\dD$,
  \item \label{o:simple}
simple curves with one endpoint in $\dD$,
  \item \label{o:plinear}
piecewise-linear simple curves
with one endpoint in $\dD$,
  \item \label{o:pathwise}
pathwise-connected sets $A$
with $A\cap\dD\neq\emptyset$,
  \end{romenumerate}
Furthermore, in \ref{o:curves}--\ref{o:plinear} we may assume that 
both endpoints of the strings are on the boundary $\dD$, but that otherwise
the strings lie in the interior $D^\circ$,
and that the
intersection representation is such that 
such that every pair of strings intersects finitely many times,
such that no three curves cross at a single point,
and such that if two strings cross, then they cross properly.
\end{lemma}

\begin{proof}
The implications 
\ref{o:plinear}$\implies$\ref{o:simple}$\implies$\ref{o:curves}
$\implies$\ref{o:pathwise} are trivial.

\ref{o:pathwise}$\implies$\ref{o:plinear}:
Suppose that $G$ has an 
intersection representation $\set{A_i}$ with
pathwise-connected sets as in \ref{o:pathwise}, and argue as in the proof of
\refL{le:strings}, now choosing $p_i\in \partial D$. (If necessary we may
add small arcs of $\dD$ to the $A_i$ in order to get all $p_i$ distinct.)
The construction above yields piecewise-linear closed curves $\gamma_i$
lying in the interior of a slightly larger disk $D'$. Let $q_i$ be the point
furthest from 
the center of $D$ where the radius through $p_i$ intersects $\gamma_i$ (such
a point exists, and is close to $p_i$); we
may then remove a small interval about $q_i$ from $\gamma_i$ and replace it
by two line segments to the boundary $\partial D'$.
This yields an intersection representation as in \ref{o:plinear}, which further
satisfies the conditions in the final statement.
\end{proof}

\end{document}